\documentclass[12pt]{amsart}
\pdfoutput=1

\usepackage[utf8]{inputenc}
\usepackage{indentfirst}
\usepackage{amsmath}
\usepackage{amsthm}
\usepackage{amssymb}
\usepackage{amsfonts}
\usepackage{microtype}
\usepackage{fancyhdr}
\usepackage{tikz-cd}
\usetikzlibrary{arrows}
\usetikzlibrary{positioning}
\usepackage{xcolor}
\usepackage[colorlinks]{hyperref}
\usepackage{mathtools}
\usepackage{thmtools}

\usepackage[
  margin=1in,
  includefoot,
  footskip=30pt,
]{geometry}

\usepackage{csquotes}

\usepackage[style=alphabetic, citestyle=alphabetic]{biblatex}
\newcommand{\A}{\mathbb{A}}

\renewcommand{\k}{\Bbbk}
\renewcommand{\O}{\mathcal{O}}
\renewcommand{\P}{\mathbb{P}}

\newcommand{\Z}{\mathbb{Z}}
\newcommand{\BL}{\mathrm{BL}}
\newcommand{\Spec}{\mathrm{Spec}\,}

\newcommand{\Hom}{\mathrm{Hom}\,}

\newcommand{\RGamma}{\mathrm{R}\Gamma\,}

\newcommand{\Dbcoh}{D^b_{\!\mathrm{coh}}}
\newcommand{\Dperf}{D\mathrm{Perf}}
\newcommand{\Dqcoh}{D_{\mathrm{QCoh}}}

\newcommand{\Tot}{\mathrm{Tot}}
\newcommand{\GL}{\mathrm{GL}}

\newcommand{\dual}{{\scriptstyle\vee}}
\newcommand{\iso}{\simeq}
\newcommand{\caniso}{\cong}

\newcommand{\codim}{\operatorname{codim}}

\newcommand{\Gr}{\operatorname{Gr}}

\declaretheoremstyle[
headformat=\NUMBER.\,\NAME\NOTE,
postheadspace=.5em,
spaceabove=6pt,
headfont=\normalfont\bfseries,
notefont=\mdseries, notebraces={(}{)},
bodyfont=\normalfont\itshape
]{plainswap}
\declaretheoremstyle[
headformat=\NUMBER.\,\NAME\NOTE,
postheadspace=.5em,
spaceabove=6pt,
headfont=\normalfont\bfseries,
notefont=\mdseries, notebraces={(}{)},
bodyfont=\normalfont
]{definitionswap}
\declaretheoremstyle[
headformat=\NAME\NOTE,
postheadspace=.5em,
spaceabove=6pt,
headfont=\normalfont\itshape,
notefont=\mdseries, notebraces={(}{)},
bodyfont=\normalfont
]{myremark}

\declaretheorem[style=plainswap, name=Theorem, sharenumber=subsection]{theorem}
\declaretheorem[style=plainswap, numberlike=theorem, name=Proposition]{proposition}
\declaretheorem[style=plainswap, numberlike=theorem, name=Lemma]{lemma}
\declaretheorem[style=plainswap, numberlike=theorem, name=Corollary]{corollary}

\theoremstyle{definition}
\declaretheorem[style=definitionswap, numberlike=theorem, name=Definition]{definition}
\declaretheorem[style=definitionswap, numberlike=theorem, name=Example]{numberedexample}
\theoremstyle{myremark}
\newtheorem*{remark}{Remark}

\theoremstyle{remark}
\newtheorem{proofstep}{Step}

\begin{document}

\title{Semiorthogonal decompositions on total spaces of tautological bundles}
\author{Dmitrii Pirozhkov}
\address{Department of Mathematics, Columbia University, New York, New York, USA}
\email{dpirozhkov@math.columbia.edu}
\date{July 4, 2018}

\begin{abstract}
Let $U$ be the tautological subbundle on the Grassmannian $\Gr(k, n)$. There is a natural morphism $\Tot(U) \to \A^n$. Using it, we give a semiorthogonal decomposition for the bounded derived category $\Dbcoh(\Tot(U))$ into several exceptional objects and several copies of $\Dbcoh(\A^n)$. We also prove a global version of this result: given a vector bundle $E$ with a regular section $s$, consider a subvariety of the relative Grassmannian $\Gr(k, E)$ of those subspaces which contain the value of $s$. The derived category of this subvariety admits a similar decomposition into copies of the base and the zero locus of $s$. This may be viewed as a generalization of the blow-up formula of Orlov, which is the case $k = 1$.
\end{abstract}

\maketitle

\section{Introduction}

Among the various invariants associated to an algebraic variety, the derived category of coherent sheaves on it is one of the richest. The main approach for describing these categories is to decompose them into simpler pieces using semiorthogonal decompositions. This idea, introduced in \cite{beil78}, was realized by Beilinson for projective spaces, and in \cite{kapranov84} for arbitrary Grassmannians.

Both of these decompositions were given as full exceptional collections, i.e.~decompositions into copies of the derived category of vector spaces. A decomposition with more complicated building blocks is used in Orlov's \cite{orlov93} result for blow-ups. For any blow-up $p\colon Y \to X$ of a smooth algebraic variety with a smooth center $Z \subset X$, Orlov described $\Dbcoh(Y)$ in terms of the categories $\Dbcoh(X)$ and $\Dbcoh(Z)$.

A special case is the blow-up of the affine space $\A^n$ in the origin. The fiber over the origin is $\P^{n-1}$, the projectivization of the original space, and the blow-up may be identified with the total space of $\O(-1)$ on this central fiber. Beilinson's description of $\Dbcoh(\P^{n-1})$ and Orlov's description of $\Dbcoh(\Tot(\O_{\P^{n-1}}(-1)))$ happen to be very similar.

In this paper we generalize this observation to all Grassmannians. Let $V$ be a vector space of dimension $n$. For an integer $k < n$ let $U$ denote the tautological subbundle on $\Gr(k, V)$. There is a morphism $\Tot(U) \to \A(V) = \A^n$ induced by the inclusion $U \subset V \otimes \O_{\Gr(k, V)}$. This is an analogue of the blow-up morphism $\Tot(\O_{\P^{n-1}}(-1)) \to \A^n$. The main result of this paper is the following theorem, which in the case $k = 1$ agrees with the Orlov's semiorthogonal decomposition.

{
\hypersetup{hidelinks}
\renewcommand{\thetheorem}{\ref{theorem on decomposition for ukn}}
\begin{theorem}
There exists a semiorthogonal decomposition:
\[
\Dbcoh(\Tot(U)) \, \caniso \, \left\langle \,\, \textstyle\binom{n-1}{k} \, \mathrm{copies \, of}\, D^b(\mathrm{Vect}), \,\,  \textstyle\binom{n-1}{k-1} \, \mathrm{copies \, of} \, \Dbcoh(\A(V))  \,\, \right\rangle .
\]
\end{theorem}
\addtocounter{theorem}{-1}
}
\noindent
See Section \ref{section on decomposition for ukn} for the precise descriptions of the components.

The decomposition is constructed from a (modified) Kapranov's collection on $\Gr(k, V)$ by pulling back some of the objects to the total space, and pushing forward the rest along the zero section.

The blow-up of a smooth subvariety in a smooth center \'etale-locally looks like the blow-up of an affine space along an affine subspace. Using the theory of relative semiorthogonal decompositions \cite{kuznetsov-basechange, conservative-descent} one may deduce the Orlov's result from the special case described above. Similarly, Theorem \ref{theorem on decomposition for ukn} may be considered as a local model for the following construction.

Let $X$ be a Cohen--Macaulay algebraic variety, $E$ a rank $n$ vector bundle on $X$ with a regular section $s$ cutting out a subvariety $Z \subset X$. Consider the closed subvariety $\Gr_s(k, E)$ of the relative Grassmannian consisting of those subspaces which contain the value of $s$. In the special case when $X = \A^n$, $E = \O^{\oplus n}$, and $s$ is the tautological section, this construction outputs precisely $\Tot_{\Gr(k, n)}(U)$ mapping to the base $\A^n$.

{
\hypersetup{hidelinks}
\renewcommand{\thetheorem}{\ref{theorem on relative decomposition for grs}}
\begin{theorem}
There exists a semiorthogonal decomposition:
\[
\Dbcoh(\Gr_s(k, E)) \, \caniso \, \left\langle \,\, \textstyle\binom{n-1}{k} \, \mathrm{copies \, of}\, \Dbcoh(Z), \,\,  \textstyle\binom{n-1}{k-1} \, \mathrm{copies \, of} \, \Dbcoh(X)  \,\, \right\rangle .
\]
\end{theorem}
\addtocounter{theorem}{-1}
}
\noindent
See Section \ref{section on relative decomposition for grs} for the precise descriptions of the components.

The derived categories of coherent sheaves on total spaces have been studied before for various vector bundles on Grassmannians, but usually in terms of tilting vector bundles rather than semiorthogonal decompositions. Recall that a vector bundle (often given by a direct sum of simpler bundles) is called \emph{tilting} if it generates the derived category and has no higher Ext's to itself. In \cite[Cor.~7.4]{TU-total} a tilting vector bundle on $\Tot(T^*_{\Gr(2, 4)})$ has been constructed. The paper \cite[Prop.~C.8]{donovan-total} demonstrates a tilting generator on $\Tot(V^\dual \otimes U)$. Also, the total space of the cotangent bundle is related to the Springer resolution of the nilpotent cone. In this context Bezrukavnikov \cite{bezrukavnikov-total} has constructed an exotic $t$-structure in $\Dbcoh(\Tot(T^*_{\Gr(k, V)}))$.

The paper is structured as follows. Section \ref{section on preliminaries} contains the basic facts we use. Section \ref{section on decomposition for ukn} contains the proof of the main Theorem \ref{theorem on decomposition for ukn}, and a simple, but very important for this paper Lemma \ref{semiorthogonality on total space}, which is a criterion of semiorthogonality for objects of a special kind on the total space of $U$. The final Section \ref{section on relative decomposition for grs} proves the global version, Theorem \ref{theorem on relative decomposition for grs}.

\subsection*{Acknowledgements} The author would like to thank Aise Johan de Jong and Alexander Kuznetsov for invaluable advice, and Raymond Cheng and Alexander Perry for interesting conversations and suggestions.


\section{Preliminaries}
\label{section on preliminaries}

\subsection{Notation}
We work over a field $\k$ of characteristic zero. All varieties in this paper are assumed to be quasiprojective over $\k$, all triangulated categories to be $\k$-linear.
All functors are assumed to be derived, with the exception of the global sections functor $\Gamma$. The (hyper)cohomology functor is always denoted by $\RGamma$\!. 

For a variety $X$ the triangulated category of perfect complexes is denoted by $\Dperf(X)$, the bounded derived category of coherent sheaves by $\Dbcoh(X)$, and the unbounded derived category of quasicoherent sheaves by $\Dqcoh(X)$.

For a vector bundle $U$ on $X$ we use $\Tot(U)$ to refer to the total space of $U$, i.e.,~$\Spec_{\!X} S^\bullet U^\dual$. For a vector space $V$ over $\k$ we use the symbol $\A(V)$ to denote the associated affine space, i.e.,~$\A(V) := \Spec S^\bullet V^\dual$.

\subsection{Derived categories of coherent sheaves}
\label{prelim on derived categories}

Let $T$ be a triangulated category. In this subsection $\mathrm{Hom}^\bullet$ refers to the graded \mbox{Hom-space}. We recall a few standard notions, see, e.g., \cite{bondal-exceptional} for details. An object $E \in T$ is called \emph{exceptional} if $\mathrm{Hom}^\bullet(E, E) = \k \cdot \mathrm{id}_E$ and for any object $A \in T$ both $\mathrm{Hom}^\bullet(A, E)$ and $\mathrm{Hom}^\bullet(E, A)$ are finite-dimensional. The usual definition omits the latter requirement as it is redundant when $T \iso \Dbcoh(X)$ for a smooth proper variety $X$, but we need the non-proper case as well.

A sequence of objects $A_1, \ldots, A_n \in T$ is \emph{semiorthogonal} if $\mathrm{Hom}^\bullet(A_j, A_i) = 0$ for any pair of indices $j > i$. If all $A_i$'s are exceptional, the sequence is called an \emph{exceptional collection}. An exceptional collection is \emph{strong} if $\mathrm{Hom}^{\bullet}(A_i, A_j) = \mathrm{Hom}^0(A_i, A_j)$ for any pair of objects. An exceptional collection is \emph{full} if for any object $B \in T$ there is at least one $A_i$ from the collection such that $\mathrm{Hom}^\bullet(B, A_i) \neq 0$.

A sequence of triangulated subcategories $\mathcal{A}_1, \ldots, \mathcal{A}_n \subset T$ is \emph{semiorthogonal} if any pair of objects $A_i \in \mathcal{A}_i$, $A_j \in \mathcal{A}_j$ is semiorthogonal whenever $j > i$. A triangulated subcategory $\mathcal{A} \subset T$ is \emph{admissible} if the inclusion functor has both left and right adjoints. Given a sequence of objects $A_1, \ldots, A_n$ the symbol $\langle \, A_1, \ldots, A_n \, \rangle$ denotes the smallest triangulated subcategory of $T$ containing all $A_i$ and closed under taking direct summands.

\begin{lemma}[{\cite[Th.~3.2]{bondal-exceptional}}]
\label{exceptional objects generate admissible subcategories}
Let $E_1, \ldots E_n \in T$ be an exceptional collection. Then $\langle \, E_1, \ldots, E_n \, \rangle$ is an admissible subcategory of $T$.
\end{lemma}

\begin{definition}[\cite{gorodentsev89}]
Let $\langle A, B \rangle$ be a semiorthogonal pair of exceptional objects. The \emph{right mutation of $A$ through $B$} is the fiber $A'$ of the coevaluation map,
\[
A' := \mathrm{Cone}(A \to \mathrm{Hom}^\bullet(A, B)^\dual \otimes B)[-1] .
\]
The \emph{mutation of the pair $\langle A, B \rangle$ to the right} is the pair $\langle B, A' \rangle$. The mutation of the pair to the left is defined analogously.
\end{definition}

\begin{proposition}[\cite{gorodentsev89}]
If $\langle A, B \rangle$ is a semiorthogonal pair of exceptional objects, then its mutation to the right $\langle B, A' \rangle$ is also a semiorthogonal pair of exceptional objects.  
\end{proposition}

\begin{definition}
Let $\langle E_1, \ldots, E_k \rangle$ be an exceptional collection. A \emph{rotation of the collection to the right} is an exceptional collection $\langle \, E_k^\prime, E_1, E_2, \ldots, E_{k-1} \, \rangle$ constructed by a sequence of left mutations of $E_k$ through $E_{k-1}$, \ldots, $E_1$. By an abuse of notation we also use this term for exceptional collections $\langle \, E_k^\prime[i], E_1, \ldots, E_{k-1} \, \rangle$ for any shift $i \in \Z$. A rotation of the collection to the left is defined analogously.
\end{definition}

The term mutation will refer to either a mutation to the left or to the right depending on the context.

%

%

The following lemma is used in the proof of the Theorem \ref{theorem on decomposition for ukn}.


\begin{lemma}
\label{adjoints for proper maps of smooth varieties}
Let $f\colon Y \to X$ be a proper map of smooth varieties. Let $\omega_f = \omega_Y \otimes f^*\omega_X^\dual$ be the relative dualizing line bundle on $Y$. Then the functor $f_!(-) := f_*( (-) \otimes \omega_f[\dim(Y) - \dim(X)])$ is a left adjoint to $f^*$.
\end{lemma}
\begin{proof}
By, e.g., \cite[Th.~3.34]{HuybFM} the functor $f^!(-) \caniso f^*(-) \otimes \omega_f[\dim(Y) - \dim(X)]$ is a right adjoint to $f_*$. The tensor multiplication by $\omega_f[\dim(Y) - \dim(X)]$ is an autoequivalence of $\Dbcoh(Y)$, so for any $A \in \Dbcoh(Y)$ and $B \in \Dbcoh(X)$ we have
\[
\begin{aligned}
\MoveEqLeft
\Hom_Y(A, f^*B) \caniso \Hom_Y(A \otimes \omega_f[\dim(Y) - \dim(X)], f^!B) \caniso  \\
& \caniso \Hom_X(f_*(A \otimes \omega_f[\dim(Y) - \dim(X)]), B) \caniso \Hom_X(f_!(A), B).
\end{aligned}
\]
Thus the functors $f_!$ and $f^*$ are adjoint to each other.
\end{proof}

\subsection{Kapranov's exceptional collection}
\label{prelim on kapranov}

Fix a vector space $V$ of dimension $n$. The main object of interest of this section is the Grassmannian $\Gr(k, V)$ of $k$-dimensional subspaces. Until the end of the paper, $U$ refers to the universal subbundle on $\Gr(k, V)$, and $Q$ to the universal quotient bundle.

The bounded derived category of coherent sheaves on $\Gr(k, V)$ has been studied by Kapranov in \cite{kapranov84} and \cite{kapranov88}. He constructed a full exceptional collection of vector bundles in this category. Along the way, Kapranov established several useful lemmas, some of which we cite separately in order to clarify the proof of Theorem \ref{theorem on decomposition for ukn}.

For each nonincreasing sequence of integers $\lambda = (\, \lambda_1 \geq \ldots \geq \lambda_k \, )$ let $L_{\lambda}(U)$ be the vector bundle on $\Gr(k, V)$ obtained by the application of the corresponding Schur functor (see, e.g., \cite{weyman}) to $U$. The sequence $\lambda$ can be thought of as a weight of $\GL(k)$. When all $\lambda_i$'s are nonnegative, we associate to $\lambda$ a Young diagram with $i$'th row of length $\lambda_i$. Our convention for Schur functors is that the vertical Young diagrams correspond to exterior powers of $U$. 

\begin{proposition}[{Pieri's formula; see, e.g., \cite[Cor.~2.3.5]{weyman}}]
\label{pieri}
Let $\lambda$ be a Young diagram, $m$ an integer. Then
\[
L_{\lambda} U \otimes \Lambda^m U \iso \bigoplus_\mu L_{\mu} U
\]
where the sum is indexed by the set of Young diagrams $\mu$ such that $\mu_i \in [\lambda_i, \lambda_i + 1]$ for all $i$, and $|\mu| = |\lambda| + m$, i.e.~$\mu$ contains $\lambda$, has $m$ additional boxes, and each row has at most one new box.
\end{proposition}

Let $B_{k, n-k}$ denote the set of Young diagrams with at most $k$ rows and at most $n-k$ columns. 

\begin{lemma}[{\cite[(3.5)]{kapranov88}}]
\label{kapranov semiorthogonality}
Let $\lambda$, $\mu$ be two diagrams in $B_{k, n-k}$. If $\lambda$ does not contain $\mu$, then
$
\Hom_{\Gr(k, V)}(L_{\lambda}(U), L_{\mu}(U)) = 0 \quad \text{and} \quad \Hom_{\Gr(k, V)}(L_{\mu^T}(Q), L_{\lambda^T}(Q)) = 0 \,\, .
$
\end{lemma}
\begin{proof}(sketch)
The tensor product $L_{\lambda} U^\dual \otimes L_{\mu} U$ is decomposed into a direct sum of Schur functors according to the Littlewood--Richardson rule. Then the Borel--Weil--Bott theorem computes the cohomology of those direct summands. It turns out that a summand $L_\eta U$ with weight $\eta$ can have nonvanishing cohomology only if each $\eta_i \leq 0$. Note that $L_\eta U$ is a direct summand of $L_\lambda U^\dual \otimes L_\mu U$ if and only if $L_\lambda U$ is a direct summand of $L_{\eta} U^\dual \otimes L_\mu U$, but for $\eta$ nonpositive this possibility is excluded by the Littlewood--Richardson formula. 
\end{proof}

\begin{theorem}[\cite{kapranov84}]
\label{kapranov full exceptional collection}
For any linear order on $B_{k, n-k}$ reversing inclusions the following sequence is a full strong exceptional collection on $\Gr(k, V)$:
\[
\Dbcoh(\Gr(k, V)) \caniso \langle \,\, \{ \, L_{\lambda}(U) \, \}_{\lambda \in B_{k, n-k}} \,\, \rangle .
\]
\end{theorem}

\begin{remark}
Since $\Gr(k, V) \caniso \Gr(n-k, V^\dual)$, the sequence $\{ \, L_{\lambda^T}(Q) \, \}_{\lambda \in B_{k, n-k}^{\mathrm{op}}}$ is also a full strong exceptional collection.
\end{remark}

\begin{lemma}[{\cite{kapranov88}}]
\label{kapranov mutations}
Let $\lambda \in B_{k, n-k}$ be a diagram, and let $B_{< \lambda}$ be the set of Young diagrams strictly contained in $\lambda$. Then
\begin{itemize}
\item there is a long exact sequence of vector bundles on $\Gr(k, V)$
\[
0 \to L_\lambda U \to F_{|\lambda| - 1} \to \ldots \to F_1 \to F_0 \to \, L_{\lambda^T} Q \, \to 0
\]
where $F_m$ is a direct sum of bundles of the form $L_\mu U$ such that $|\mu| = m$ and $\mu \in B_{< \lambda}$.
\item for any linear order of $B_{< \lambda}$ reversing the inclusions, the rotation to the left of the collection $\langle \, L_\lambda U, \, \{\, L_\mu U \,\}_{\mu \in B_{< \lambda}} \, \rangle$ is the collection $\langle \, \{\, L_\mu U \,\}_{\mu \in B_{< \lambda}} \, , L_{\lambda^T} Q \, \rangle$.
\item for any linear order of $B_{< \lambda}$ reversing the inclusions, the following two exceptional collections can be transformed one into another by a sequence of mutations:
\[
\langle \,\, L_\lambda U, \, \{\, L_\mu U \,\}_{\mu \in B_{< \lambda}} \,\, \rangle \quad \leftrightsquigarrow \quad \langle \,\, \{\, L_{\mu^T} Q \,\}_{\mu \in B_{< \lambda}^{\mathrm{op}}}, \, L_{\lambda^T} Q \,\, \rangle
\]
In particular, they generate the same subcategory of $\Dbcoh(\Gr(k, V))$.
\end{itemize}
\end{lemma}

\begin{proof}
By \cite[(3.3)]{kapranov88} there is a resolution of $L_{\lambda^T} Q$ given by the complex whose $p$'th term is a direct sum of vector bundles
\[
\bigoplus_{i = 0}^\infty
\bigoplus_{\substack{\alpha \in B_{k, n-k} \\ |\alpha| = i - p}} 
\mathrm{R}^i\Gamma\left(\Gr(k, V), \, L_{\lambda^T} Q \otimes L_{\alpha^T} Q^\dual\right) \, \otimes \, L_{\alpha} U .
\]
The cohomology term is isomorphic to $\mathrm{R}^i\Hom(L_{\alpha^T} Q, L_{\lambda^T} Q)$. By Lemma \ref{kapranov semiorthogonality} it can be nonzero only when $\alpha \subset \lambda$, and the strongness property from Theorem \ref{kapranov full exceptional collection} implies that $i$ must be zero. Thus the complex is concentrated in degrees $[\, -|\lambda|, 0 \, ]$ and the $(-|\lambda|)$'th term is equal to $L_\lambda U$, as required.

To prove the second part of the statement, let $F_\bullet \in \Dbcoh(\Gr(k, V))$ be the complex
\[
F_{|\lambda| - 1} \to \ldots \to F_1 \to F_0
\]
as in the sequence above, with $F_{|\lambda| - 1}$ put in degree 0. Then there is an exact triangle in the category $\Dbcoh(\Gr(k, V))$:
\[
L_{\lambda} U \to F_\bullet \xrightarrow{\text{cone}} L_{\lambda^T} Q \, [1-|\lambda|]
\]
By definition $F_\bullet \in \{\, L_\mu U \,\}_{\mu \in B_{< \lambda}}$, and by \cite[Prop.~2.2]{kapranov84} the vector bundle $L_{\lambda^T} Q$ is left-orthogonal to that subcategory.
These properties characterize the mutation triangles, so the second part is proved. The iteration of this argument proves the third part of the statement as well.
\end{proof}

\subsection{Relative semiorthogonal decompositions}
\label{linear subcategories}

Until the end of this subsection, fix a quasiprojective variety $X$. The notions described below allow us to work with the derived categories of sheaves on varieties over $X$ locally on the base. In Section \ref{section on relative decomposition for grs} this machinery is applied to prove Theorem \ref{theorem on relative decomposition for grs} by using Theorem \ref{theorem on decomposition for ukn} as a local model.

There are many possible levels of generality. The ideas originate from \cite[(2.7)--(2.8)]{kuznetsov06} and \cite[(2.8)]{hpd}, but these sources work with smooth varieties, while we need locally complete intersections. Our Definition \ref{definition of relative sod} of a relative semiorthogonal decomposition has the least generality sufficient for the proof of Theorem \ref{theorem on relative decomposition for grs}. It is a special case of situations  considered in \cite{kuznetsov-basechange} and \cite{conservative-descent}. These two papers establish, respectively, the base change theorem (Theorem \ref{relative base change of sod}) and the faithfully flat descent theorem (Theorem \ref{relative descent of sod}) for semiorthogonal decompositions.

The settings of \cite{kuznetsov-basechange} and \cite{conservative-descent} are slightly different, so in order to use both results at the same time we need some bookkeeping. Thus we start by proving several lemmas about locally complete intersection morphisms. Recall the standard definition.

\begin{definition}
A morphism $f\colon Y_1 \to Y_2$ of varieties is a \emph{locally complete intersection morphism} if any closed point $y \in Y_1$ has an open neighborhood $U \subset Y_1$ and a factorization of the restriction $f|_U$ as $\pi \circ i$, where $i\colon U \to P$ is a regular closed embedding and $\pi\colon P \to Y_2$ is a smooth morphism.
\end{definition}

For a variety $Y$ denote by $\Dqcoh^{[p, q]}(Y)$ the subcategory of $\Dqcoh(Y)$ of complexes $G$ such that $\mathcal{H}^i(G) = 0$ for $i$ outside of the segment $[p, q]$.

\begin{definition}
Let $Y_1, Y_2$ be two varieties. A functor $\Phi\colon \Dqcoh(Y_1) \to \Dqcoh(Y_2)$ has \emph{finite cohomological amplitude} if there exist integers $a, b$ such that for any $p, q$ the restriction of $\Phi$ to $\Dqcoh^{[p, q]}(Y_1)$ has image in $\Dqcoh^{[p-a, q+b]}(Y_2)$.
\end{definition}

The following lemma is a combination of several standard results.

\begin{lemma}
\label{proper lci properties}
Let $f\colon Y_1 \to Y_2$ be a proper locally complete intersection morphism of varieties. Then 
\begin{enumerate}
\item The functor $f_*\colon \Dqcoh(Y_1) \to \Dqcoh(Y_2)$ has a right adjoint $f^!$;
\item The functor $f^*\colon \Dqcoh(Y_2) \to \Dqcoh(Y_1)$ has a left adjoint  $f_!$;
\item \label{proper lci functors amplitude} The functors $f_!$, $f^*$, $f_*$, $f^!$  have finite cohomological amplitude;
\item \label{proper lci functors dbcoh} The functors $f_!$, $f^*$, $f_*$, $f^!$ restrict to functors between $\Dbcoh(Y_1)$ and $\Dbcoh(Y_2)$;
\item \label{proper lci functors dperf} The functors $f_!$, $f^*$, $f_*$, $f^!$ restrict to functors between $\Dperf(Y_1)$ and $\Dperf(Y_2)$.
\end{enumerate}
\end{lemma}
\begin{proof}
Recall that locally complete intersection morphisms are perfect \cite[Tag~069H]{stacks-project}, so by, e.g., \cite[Th.~4.8.1\,and\,Th.~4.9.4]{lipman09}, the functor $f^!$ exists and $f^!(-) = f^*(-) \otimes f^! \O_{Y_2}$. Furthermore, the object $f^! \O_{Y_2}$ is a shift of a line bundle \cite[Tag~0E4B]{stacks-project}, so it is an invertible object of $\Dqcoh(Y_1)$. Hence for any $A \in \Dqcoh(Y_1), B \in \Dqcoh(Y_2)$ we get
\[
\Hom_{Y_1}(A, f^* B) \caniso \Hom_{Y_1}(A \otimes f^! \O_{Y_2}, f^*B \otimes f^! \O_{Y_2}) \caniso \Hom_{Y_2}(f_*(A \otimes f^! \O_{Y_2}), B).
\]
Thus the functor $f_!(A) := f_*(A \otimes f^! \O_{Y_2})$ is the left adjoint to $f^*$.

The functors $f_!$ and $f^!$ are built out of $f_*$, $f^*$, and the tensor multiplication by a perfect complex $f^! \O_{Y_2}$. Therefore it is enough to prove (\ref{proper lci functors amplitude}), (\ref{proper lci functors dbcoh}), and (\ref{proper lci functors dperf}) for the functors $f_*$ and $f^*$.

The inclusion $f^*(\Dperf(Y_2)) \subset \Dperf(Y_1)$ holds for any morphism. Since $f$ is a  locally complete intersection morphism, $f^*(\Dbcoh(Y_2)) \subset \Dbcoh(Y_1)$ and $f^*$ has finite cohomological amplitude. 
The inclusion $f_*(\Dbcoh(Y_1)) \subset \Dbcoh(Y_2)$ holds for any proper morphism of varieties. The dimensions of fibers of $f$ is bounded, so $f_*$ is of finite cohomological amplitude. Since $f$ is proper and perfect, by \cite[Prop.~2.1]{lipman-neeman} we have $f_*(\Dperf(Y_1)) \subset \Dperf(Y_2)$. Thus the items (\ref{proper lci functors amplitude}), (\ref{proper lci functors dbcoh}), and (\ref{proper lci functors dperf}) hold.
\end{proof}

\begin{remark}
Lemma \ref{proper lci properties} subsumes Lemma \ref{adjoints for proper maps of smooth varieties} since 
any morphism between smooth varieties is a locally complete intersection morphism \cite[Tag~0E9K]{stacks-project}.
\end{remark}

\begin{lemma}[{\cite[Prop.~3.5]{conservative-descent}}]
\label{lci fourier-mukai properties}
Let $Y_1, Y_2$ be two varieties
with projective maps $\pi_i\colon Y_i \to X$ 
such that the projection morphisms $p_i\colon Y_1 \times_X Y_2 \to Y_i$ are locally complete intersection morphisms. Let $K \in \Dperf(Y_1 \times_X Y_2)$ be a perfect complex. Let $\Phi_{K}\colon \Dqcoh(Y_1) \to \Dqcoh(Y_2)$ be the functor defined by the formula $G \mapsto p_{2\, *}(p_1^* G \otimes K)$. Then
\begin{enumerate}
\item \label{lci fourier-mukai adjoints} there exist perfect complexes $K^\prime, K^{\prime \prime} \in \Dperf(Y_2 \times_X Y_1)$ such that the functors $\Phi_{K^\prime}, \, \Phi_{K^{\prime \prime}}\colon \Dqcoh(Y_2) \to \Dqcoh(Y_1)$ are left and right adjoint functors to $\Phi_{K}$;
\item \label{lci fourier-mukai amplitude} The functors $\Phi_{K}$, $\Phi_{K^{\prime}}$, $\Phi_{K^{\prime \prime}}$ are of finite cohomological amplitude.
\item The functors $\Phi_{K}$, $\Phi_{K^{\prime}}$, $\Phi_{K^{\prime \prime}}$ restrict to functors between $\Dbcoh(Y_1)$ and $\Dbcoh(Y_2)$;
\item The functors $\Phi_{K}$, $\Phi_{K^{\prime}}$, and $\Phi_{K^{\prime \prime}}$ restrict to functors between the categories $\Dperf(Y_1)$ and $\Dperf(Y_2)$.
\end{enumerate}
\end{lemma}
\begin{proof}
Using the explicit descriptions of $p_{1 \, !}$ and $p_2^!$ from Lemma \ref{proper lci properties} it is easy to see that the objects $K^\prime := K^\dual \otimes p_1^!\O_{Y_1}$ and $K^{\prime \prime} := K^\dual \otimes p_2^!\O_{Y_2}$ give (\ref{lci fourier-mukai adjoints}). The other three statements follow directly from Lemma \ref{proper lci properties}.
\end{proof}

\begin{remark}
In the terminology of \cite{conservative-descent}, 
Lemma \ref{proper lci properties} shows that the functor $\Phi_{K}$ and its adjoints from Lemma \ref{lci fourier-mukai properties} are \emph{Fourier--Mukai transforms over $X$} as in \cite[Def.~3.3]{conservative-descent}.
\end{remark}

\begin{definition}
\label{definition of relative sod}
Let $Y$, $\{ X_i \}_{i \in [1, n]}$ be varieties 
projective over $X$
such that the projection maps $Y \times_X X_i \to Y$ and $Y \times_X X_i \to X_i$ are locally complete intersection morphisms. Let also $\{ K_i \in \Dperf(Y \times_X X_i) \}_{i \in [1, n]}$ be a set of perfect complexes. This collection is called a \emph{relative semiorthogonal decomposition} of $Y$ over $X$ if
\begin{enumerate}
\item each $\Phi_{K_i}\colon \Dbcoh(X_i) \to \Dbcoh(Y)$ is a fully faithful embedding onto an admissible subcategory $\mathcal{A}_{i, \, \mathrm{coh}} \subset \Dbcoh(Y)$;
\item the sequence $\langle \mathcal{A}_{1, \, \mathrm{coh}}, \ldots, \mathcal{A}_{n, \, \mathrm{coh}} \rangle \subset \Dbcoh(Y)$ is a semiorthogonal decomposition.
\end{enumerate}
\end{definition}

\begin{remark}
In the language of \cite{kuznetsov-basechange} this decomposion is strong, $X$-linear, and its projection functors $\Dbcoh(Y) \to \mathcal{A}_{i, \, \mathrm{coh}}$ have finite cohomological amplitude by Lemma \ref{proper lci properties} (\ref{lci fourier-mukai amplitude}).
\end{remark}

\begin{theorem}[{\cite[Th.~5.6]{kuznetsov-basechange}}]
\label{relative base change of sod}
Let $(Y, \{ X_i, K_i \}_{i \in [1, n]})$ be a relative semiorthogonal decomposition of $Y$.  Let $\phi\colon X^\prime \to X$ be a flat map. Then the base change $(Y^\prime, \{ X_{i}^\prime, K_{i}^\prime \}_{i \in [1, n]})$ of this data forms a relative semiorthogonal decomposition of $Y^\prime$ over $X^\prime$.
\end{theorem}

\begin{remark}
The condition on the base change map $\phi\colon X^\prime \to X$ in \cite[Th.~5.6]{kuznetsov-basechange} is that it is faithful with respect to the map $Y \to X$ (see, e.g., \cite[Sec.~2.4]{kuznetsov-basechange}), not necessarily flat. But we made it a part of Definition \ref{definition of relative sod} that certain projection maps are locally complete intersections. It is not clear that this condition is preserved unless the morphism $\phi$ is flat.
\end{remark}

\begin{theorem}[{\cite[Th.~6.1\,and\,6.2]{conservative-descent}}]
\label{relative descent of sod}
Let $Y, \{ X_i \}_{i \in [1, n]}$ be varieties 
projective over $X$
such that the projection maps $Y \times_X X_i \to Y$ and $Y \times_X X_i \to X_i$ are locally complete intersection morphisms. Let $\{ K_i \in \Dperf(Y \times_X X_i) \}_{i \in [0, n]}$ be a set of perfect complexes. Given a faithfully flat morphism $\phi\colon X^\prime \to X$ denote by $Y^\prime$, $X_{i}^\prime$, $K_{i}^\prime$ the base change of these objects along $\phi$. If $(Y^\prime, \{ X_i^\prime, \, K_i^\prime \}_{i \in [1, n]})$ is a relative semiorthogonal decomposition of $Y^\prime$ over~$X^\prime$, then $(Y, \{ X_i, K_i \}_{i \in [1, n]})$ is a relative semiorthogonal decomposition of $Y$ over $X$.
\end{theorem}

The relative semiorthogonal decompositions induce semiorthogonal decompositions not only for $\Dbcoh(Y)$, but also for other categories such as $\Dqcoh(Y)$ and $\Dperf(Y)$.

\begin{lemma}[\cite{kuznetsov-basechange}]
\label{relative decompositions are for all categories}
Let $(Y, \{ X_i, K_i \}_{i \in [1, n]})$ be a relative semiorthogonal decomposition of $Y$ over $X$. Then
\begin{enumerate}
\item the functors $\Dqcoh(X_i) \to \Dqcoh(Y)$ and $\Dperf(X_i) \to \Dperf(Y)$ induced by $\Phi_{K_i}$ are fully faithful embeddings onto admissible subcategories $\mathcal{A}_{i, \, \mathrm{qc}}$ and $\mathcal{A}_{i, \, \mathrm{perf}}$ of $\Dqcoh(Y)$ and $\Dperf(Y)$ respectively;
\item the sequences $\langle \mathcal{A}_{1, \, \mathrm{qc}}, \ldots, \mathcal{A}_{n, \, \mathrm{qc}} \rangle \subset \Dqcoh(Y)$ and $\langle \mathcal{A}_{1, \, \mathrm{perf}}, \ldots, \mathcal{A}_{n, \, \mathrm{perf}} \rangle \subset \Dperf(Y)$ are semiorthogonal decompositions.
\end{enumerate}
\end{lemma}
\begin{proof}
The semiorthogonal decomposition of $\Dbcoh(Y)$ into $\mathcal{A}_{i, \, \mathrm{coh}}$ induces decompositions for the categories $\Dqcoh(Y)$ and $\Dperf(Y)$ by \cite[Prop.~4.2, Prop.~4.1]{kuznetsov-basechange}. The components are equal to the images of $\Dqcoh(X_i)$ and $\Dperf(X_i)$ under $\Phi_{K}$ by \cite[Th.~6.2,~Rem.~6.3]{kuznetsov-basechange}.
\end{proof}

\section{Total space of the tautological bundle}
\label{section on decomposition for ukn}

Let $V$ be a vector space of dimension $n$, and let $U$ and $Q$ be the tautological subbundle and quotient bundle on $\Gr(k, V)$. Denote by $\Tot(U)$ the total space of $U$ with the projection map $\pi\colon \Tot(U) \to \Gr(k, V)$. Given a full exceptional collection
\[
\Dbcoh(\Gr(k, V)) = \langle E_1, \ldots, E_N \rangle
\]
it is not hard to show that the pullbacks $\pi^*E_i$ generate $\Dbcoh(\Gr(k, V))$. But the semiorthogonality of the sequence usually breaks after the pullback, as shown in Lemma \ref{semiorthogonality on total space}. The goal of this section is to propose a way to work around this problem.

For inspiration, consider the case $k = 1$. Then $\Gr(1, V)$ is the projective space $\P(V)$, the tautological bundle is $\O_{\P(V)}(-1)$, and the total space $\Tot(\O(-1))$ is naturally  identified with the blow-up  $\BL(\A(V), \{ 0 \} )$ of the affine space $\A(V)$ at the origin. The derived category of sheaves on the blow-up admits the following description.

\begin{theorem}[\cite{orlov93}]
Let $j\colon \P(V) \to \BL(\A(V), \{ 0 \})$ be the inclusion of the exceptional divisor, and let $p\colon \BL(\A(V), \{ 0 \} ) \to \A(V)$ be the blowing up map. Then there is a semiorthogonal decomposition
\[
\Dbcoh(\BL(\A(V), \{ 0 \} )) = \langle \,\, j_* \O_{\P(V)}(-n), \, \ldots, \, j_* \O_{\P(V)}(-1), \, p^*\Dbcoh(\A(V)) \,\, \rangle.
\]
\end{theorem}

The proof of this theorem relies on the properties of the blow-up morphism $p$, but its conclusion may be stated using only $\pi$ and $j$ as follows. Since $\Dbcoh(\A(V))$ is generated by the structure sheaf, the last subcategory coincides with $\langle \, \O_{\BL(\A(V), \{ 0 \} )} \, \rangle$. Start with a Beilinson full exceptional collection of the projective space \cite{beil78}:
\[
\Dbcoh(\P(V)) = \langle \,\, \O_{\P(V)}(-n), \, \ldots, \, \O_{\P(V)}(-1), \, \O_{\P(V)} \,\, \rangle .
\]
To construct a semiorthogonal decomposition on $\Tot(\O(-1))$, take the pullback $\pi^*$ of the last object of this sequence, and then the pushforwards of the remaining $n-1$ objects along the zero section embedding $j\colon \P(V) \to \Tot(\O_{\P(V)}(-1))$.

This suggests a general approach: given a total space of a vector bundle and a semiorthogonal decomposition of its base, try pulling back some components and pushing forward the rest of them. Usually this will not be good enough. But if we choose the decomposition of the base carefully, then this strategy works for $\Gr(k, V)$ and $\Tot(U)$. 

We start by a couple of lemmas.
Recall that in our conventions the symbol $\Hom_{\Gr(k, V)}(-, -)$ means the derived functor and it takes values in $D^b(\mathrm{Vect})$.

\begin{lemma}
\label{semiorthogonality of pushforwards on total space}
Let $V, U, Q$ be as above. Let $j\colon \Gr(k, V) \to \Tot(U)$ be the zero section. For any two objects $E, F \in \Dbcoh(\Gr(k, V))$ the following conditions are equivalent:
\begin{enumerate}
\item \label{push semiorthogonality 1} $\langle \, j_* E, \, j_* F \, \rangle$ is a semiorthogonal pair on $\Tot(U)$;
\item \label{push semiorthogonality 2} for any integer $m$, $\Hom_{\Gr(k, V)}(F \otimes \Lambda^m U^\dual, E) = 0$.
\end{enumerate}
\end{lemma}
\begin{proof}
Let $\pi\colon \Tot(U) \to \Gr(k, V)$ be the projection morphism. Since $j$ is a section of $\pi$, for any object $F \in \Dbcoh(\Gr(k, V))$ there is an isomorphism $j^* j_* F \caniso F \otimes j^* j_* \O_{\Gr(k, V)}$.

The zero section of $\Tot(U)$ is cut out by the tautological section of the vector bundle $\pi^*U$ on $\Tot(U)$, which is regular. Hence there is a locally free Koszul resolution for $j_*\O_{\Gr(k, V)}$:
\[
0 \to \pi^*\Lambda^k U^\dual \to \ldots \to \pi^* U^\dual \to \O_{\Tot(U)} \to 0.
\]
It may be used to calculate the derived pullback $j^* j_* \O_{\Gr(k, V)}$. The Koszul differentials vanish along the zero section, so there is a direct sum decomposition
$
j^* j_* \O_{\Gr(k, V)} \caniso \bigoplus\nolimits_{i=0}^{k} \Lambda^i U^\dual [i].
$
From this description we get an isomorphism
\[
\begin{aligned}
\MoveEqLeft
\Hom_{\Tot(U)}(j_* F, \, j_* E) \caniso \Hom_{\Gr(k, V)}(j^* j_* F, E) \caniso \bigoplus\nolimits_{i=0}^{k} \Hom_{\Gr(k, V)}(F \otimes \Lambda^i U^\dual, E) [-i].
\end{aligned}
\]
Thus the conditions (\ref{push semiorthogonality 1}) and (\ref{push semiorthogonality 2}) are equivalent.
\end{proof}

\begin{lemma}
\label{tautological subcategories}
Let $V, U, Q$ be as above. For any integer $m \geq 0$ the subcategories 
$\langle \, S^i Q \, \rangle_{i \leq m}$ 
and
$\langle \, \Lambda^i U \, \rangle_{i \leq m}$
of $\Dbcoh(\Gr(k, V))$ coincide. Similarly, 
$\langle \, S^i U \, \rangle_{i \leq m} = \langle \, \Lambda^i Q \, \rangle_{i \leq m}$. Same for the dual pair $(Q^\dual, U^\dual)$.
\end{lemma}
\begin{proof}
The tautological short exact sequence
\[
0 \to U \to V \otimes \O_{\Gr(k, V)} \to Q \to 0
\]
induces the Koszul long exact sequences
for any $i$:
\[
\begin{aligned}
& 0 \to \Lambda^i U \to \ldots \to U \otimes S^{i-1} V \to S^i V \otimes \O_{\Gr(k, V)} \to S^i Q \to 0 \\
& 0 \to \Lambda^i U \to \Lambda^i V \otimes \O_{\Gr(k, V)} \to \Lambda^{i-1} V \otimes Q \to \ldots \to S^i Q \to 0
\end{aligned}
\]
They show that, for any $m$, $S^m Q$ lies in the subcategory $\langle \, \Lambda^i U \, \rangle_{i \leq m}$, while $\Lambda^m U$ lies in the subcategory $\langle \, S^i U \, \rangle_{i \leq m}$. This proves the first statement. The other statements follow analogously.
\end{proof}
\begin{remark}
In particular, the subcategory $\langle \, S^i U \, \rangle_{0 \leq i \leq n-k}$ contains all symmetric powers of $U$. Indeed, for any $m > n-k$
\[
S^m U \in \langle \, S^i U \, \rangle_{0 \leq i \leq m} = \langle \, \Lambda^i Q \, \rangle_{0 \leq i \leq m} = \langle \, \Lambda^i Q \, \rangle_{0 \leq i \leq n-k} = \langle \, S^i U \, \rangle_{0 \leq i \leq n-k}.
\]
\end{remark}

\begin{lemma}
\label{semiorthogonality on total space}
Let $V, U, Q$ be as above. Denote by $\pi$ and $j$ the projection and the zero section of $\Tot(U)$, by $\pi^\prime$ and $j^\prime$ same for $\Tot(Q^\dual)$. For any two objects $E, F \in \Dbcoh(\Gr(k, V))$ the following conditions are equivalent:
\begin{enumerate}
\item \label{push-pull semiorthogonality 1} $\langle \, \pi^*E, \, \pi^*F \, \rangle$ is a semiorthogonal pair on $\Tot(U)$;
\item \label{push-pull semiorthogonality 2} for any integer $m$, $\Hom_{\Gr(k, V)}(F, E \otimes S^m U^\dual) = 0$;
\item \label{push-pull semiorthogonality 3} for any integer $m$, $\Hom_{\Gr(k, V)}(F \otimes \Lambda^m Q, E) = 0$;
\item \label{push-pull semiorthogonality 4} $\langle \, j_*^\prime E, \, j_*^\prime F \, \rangle$ is a semiorthogonal pair on $\Tot(Q^\dual)$;
\end{enumerate}
\end{lemma}

\begin{proof}
The morphism $\pi\colon \Tot(U) \to \Gr(k, V)$ is an affine morphism given by the relative spectrum of the symmetric algebra $S^\bullet U^\dual$, so by construction
$\pi_* \O_{\Tot(U)} \caniso \oplus_{i=0}^{\infty} S^i U^\dual$. Hence
\[
\begin{aligned}
\MoveEqLeft
\Hom_{\Tot(U)}(\pi^*F, \pi^*E) \caniso \Hom_{\Gr(k, V)}(F, \pi_* \pi^* E) \caniso \\
& \caniso \Hom_{\Gr(k, V)}(F, E \otimes \pi_*\O_{\Tot(U)}) \caniso \Hom_{\Gr(k, V)}(F, E \otimes \left(\oplus_{i=0}^{\infty} S^i U^\dual\right)).
\end{aligned}
\]

The functor $\Hom(F, -)$ commutes with direct sums since any object of $\Dbcoh(\Gr(k, V))$ is compact (see, e.g., \cite{thomason-trobaugh}). Therefore there is an isomorphism
\[
\Hom_{\Tot(U)}(\pi^*F, \pi^*E) \caniso \bigoplus_{i=0}^{\infty} \Hom_{\Gr(k, V)}(F, E \otimes S^i U^\dual).
\] 
This proves the equivalence of (\ref{push-pull semiorthogonality 1}) and (\ref{push-pull semiorthogonality 2}). 

The condition (\ref{push-pull semiorthogonality 2}) implies that 
$\Hom_{\Gr(k, V)}(F, E \otimes A) = 0$
for any object $A \in \langle \, S^i U^\dual \, \rangle_{i \in \Z_{\geq 0}}$, as derived tensor products and $\Hom_{\Gr(k, V)}(F, -)$ commute with taking cones. So by Lemma~\ref{tautological subcategories} the points (\ref{push-pull semiorthogonality 2}) and (\ref{push-pull semiorthogonality 3}) are equivalent. The equivalence between (\ref{push-pull semiorthogonality 3}) and (\ref{push-pull semiorthogonality 4}) is the statement of Lemma~\ref{semiorthogonality of pushforwards on total space} applied on the Grassmannian $\Gr(n-k, V^\dual)$.
\end{proof}

\begin{remark}
There is another way to see a relation between (\ref{push-pull semiorthogonality 1}) and (\ref{push-pull semiorthogonality 3}). Let $p\colon \Tot(U) \to \A(V)$ be the map induced by the embedding $U \to V \otimes \O_{\Gr(k, V)}$. The condition (\ref{push-pull semiorthogonality 1}) is equivalent to the vanishing of
\[
\Hom_{\Tot(U)}(\pi^* F, \pi^* E) \caniso \RGamma(\Tot(U), \pi^* F^\dual \otimes \pi^* E) \caniso \RGamma(\A(V), p_*(\pi^* F^\dual \otimes \pi^* E))  
\]
As $\A(V) \iso \A^n$ is affine, this vanishes if and only if $p_*(\pi^* F^\dual \otimes \pi^* E)$ is zero. The fiber at the origin of this complex is, up to shifts, the sum of $\Hom$'s from the condition (\ref{push-pull semiorthogonality 3}).
\end{remark}

Now we are ready to prove the main theorem.

\begin{theorem}
\label{theorem on decomposition for ukn}
Let $V$ be an $n$-dimensional vector space, $\pi\colon \Tot(U) \to \Gr(k, V)$  the total space of the tautological bundle, $j\colon \Gr(k, V) \to \Tot(U)$ its zero section, $p\colon \Tot(U) \to \A(V)$ the morphism induced by the inclusion $U \subset \O_{\Gr(k, V)} \otimes V$. Let $B_{k, n-k}^\prime \subset B_{k, n-k}$ denote the subset of Young diagrams with exactly $k$ rows, equipped with any order reversing the inclusions. Then the following is a semiorthogonal decomposition of $\Tot(U)$:
\[
\label{decomposition for ukn}
\tag{$\star$}
\begin{aligned}
\Dbcoh(\Tot(U)) = \langle \,\,\,
& \{ \,\, j_* L_{\lambda}(U) \,\, \}_{\lambda \in B_{k, n-k}^\prime} \,\, , \\
& \{ \,\, \pi^* L_{\mu^T}(Q) \otimes p^*\Dbcoh(\A(V)) \,\, \}_{\mu \in B_{k, n-k} \setminus B_{k, n-k}^\prime} \,\,\, \rangle
\end{aligned}
\]
Here for any $\lambda \in B_{k, n-k}^\prime$ the object $j_* L_\lambda(U)$ is exceptional, and for any $\mu \in B_{k, n-k} \setminus B_{k, n-k}^\prime$ the object $\pi^* L_{\mu^T}(Q)$ generates an admissible subcategory equivalent to $\Dbcoh(\A^n)$.
\end{theorem}

\begin{proof}
By Theorem \ref{kapranov full exceptional collection} the sequence of vector bundles
\[
\langle \, \{ \, L_\lambda U \, \}_{\lambda \in B_{k, n-k}^\prime}, \,\, \{ \, L_\mu U \, \}_{\mu \in B_{k, n-k} \setminus B_{k, n-k}^\prime} \, \rangle
\]
forms a full exceptional collection on $\Gr(k, V)$. The mutations from the last part of Lemma~\ref{kapranov mutations}, applied to the subsequence $\{ \, L_\mu U \, \}$, produce a full exceptional collection on $\Gr(k, V)$:
\[
\label{decomposition for grkn}
\tag{$\star\star$}
\begin{aligned}
\Dbcoh(\Gr(k, V)) = \, \langle \,\,\,
& \{ \,\, L_{\lambda}(U) \,\, \}_{\lambda \in B_{k, n-k}'} \,\, , \\
& \{ \,\, L_{\mu^T}(Q) \,\, \}_{\mu \in B_{k, n-k} \setminus B_{k, n-k}'} \,\,\, \rangle.
\end{aligned}
\]

Call the first row of (\ref{decomposition for grkn}) the ``$j_*$-part'' and the second row the ``$\pi^*$-part''. The properties of the collection (\ref{decomposition for grkn}) are used in the rest of the proof to establish the existence of (\ref{decomposition for ukn}). The proof is divided into several steps.

\begin{proofstep}
The semiorthogonality of (\ref{decomposition for ukn}) is, by Lemmas \ref{semiorthogonality of pushforwards on total space} and \ref{semiorthogonality on total space}, equivalent to the following three statements, where $\langle \, E, F \, \rangle$ is a semiorthogonal pair from (\ref{decomposition for grkn}).

\begin{itemize}
\item[(a)] If $E$, $F$ are vector bundles from the $j_*$-part, then for any $m \in \Z$
\[
\Hom_{\Gr(k, V)}(F \otimes \Lambda^m U^\dual, E) = 0.
\]
\item[(b)] If $E$ is a vector bundle from the $j_*$-part and $F$ is from the $\pi^*$-part, then
\[
\Hom_{\Tot(U)}(\pi^*F, j_*E) = 0.
\]
\item[(c)] If $E$, $F$ are vector bundles from the $\pi^*$-part, then for any $m \in \Z$
\[
\Hom_{\Gr(k, V)}(F \otimes \Lambda^m Q, E) = 0.
\]
\end{itemize}

Part (b) follows from the semiorthogonality of (\ref{decomposition for grkn}), as 
\[
\Hom_{\Tot(U)}(\pi^*F, j_*E) = \Hom_{\Gr(k, V)}(F, E) = 0.
\]
Part (a) is proved in Step 2 and Part (c) is proved in Step 3.
\end{proofstep}

\begin{proofstep}[semiorthogonality of pushforwards]
If $E, F$ is a semiorthogonal pair from the $j_*$-part, then $F \iso L_\lambda U$ for some $\lambda \in B_{k, n-k}^\prime$, and the Young diagram of $E$ is not contained in $\lambda$. By Lemma \ref{kapranov semiorthogonality} it is enough to show that for any $m > 0$ the tensor product $L_\lambda(U) \otimes \Lambda^{m}U^\dual$ splits into irreducible factors, all of which correspond to Young diagrams strictly inside $\lambda$. This would imply that all direct summands of $F \otimes \Lambda^{m} U^\dual$ are still semiorthogonal to $E$.

 Let $\lambda_i$ be the length of $i$'th row of $\lambda$. Note that
\[
L_\lambda(U) \otimes \Lambda^m U^\dual \caniso L_\lambda(U) \otimes \det(U)^{-1} \otimes \Lambda^{k-m} U \caniso L_{\lambda_1 - 1, \ldots, \lambda_k - 1} U \otimes \Lambda^{k-m} U.
\]
The twist decreases each $\lambda_i$ by one. As $\lambda \in B_{k, n-k}'$, each $\lambda_i \geq 1$, so the twist still corresponds to a Young diagram. By Pieri's formula \ref{pieri}, to get any irreducible summand of the tensor product, after the twist we increase some of $\lambda_i$'s by one. Thus, at best, for $\det(U)^{-1} \otimes \Lambda^k U$, we recover $\lambda$ back, and in all other cases the diagram is strictly smaller than $\lambda$. So the statement in (a) is proved. This argument also proves that
\[
\Hom_{\Tot(U)}(j_* L_\lambda(U), j_* L_\lambda(U)) \caniso \Hom_{\Gr(k, V)}(L_\lambda(U), L_\lambda(U)) = \k \cdot \mathrm{id}_{L_\lambda U}.
\]

The object $j_* L_\lambda(U)$ has proper support, so it is exceptional, and by Lemma \ref{exceptional objects generate admissible subcategories} the subcategory $\langle \, j_* L_\lambda(U) \, \rangle$ is admissible for any $\lambda \in B_{k, n-k}'$.
\end{proofstep}

\begin{proofstep}[semiorthogonality of pullbacks]
Part (c) is similar. If $E, F$ is a semiorthogonal pair from the $\pi^*$-part, then $F \iso  L_{\mu^T} Q$ for some $\mu \in B_{k, n-k} \setminus B_{k, n-k}^\prime$ and the Young diagram of $E$ does not contain $\mu$. By Lemma \ref{kapranov semiorthogonality} it is enough to show that for any $\mu \in B_{k, n-k} \setminus B_{k, n-k}'$ the tensor product $L_{\mu^T}(Q) \otimes \Lambda^m(Q)$ splits into irreducible factors, all of which correspond to Young diagrams containing $\mu^T$, but still lying inside the $(n-k) \times k$-rectangle.

Let $\mu^T_i$ be the length of the $i$'th column of $\mu$. By definition each $\mu^T_i \leq k-1$ and then Pieri's formula for the weights of irreducible summands of $L_{\mu^T}(Q) \otimes \Lambda^m(Q)$ increases some of $\mu^T_i$ by 1. The new weights are still inside the box, so (c) is satisfied and the collection is semiorthogonal.
\end{proofstep}

\begin{proofstep}[endomorphisms of pullbacks, admissibility]
Unlike the case (a), the argument does not immediately describe the endomorphism algebra of $\pi^* L_{\mu^T}(Q)$, but it is not very far from it. Consider the canonical map 
\[
\phi\colon \O_{\Gr(k, V)} \to L_{\mu^T}(Q)^\dual \otimes L_{\mu^T}(Q).
\]
We claim that $\pi^*\phi$ induces an isomorphism on cohomology of pullbacks. Equivalently, for any integer $m$, the map $\phi \otimes S^m U^\dual$ induces an isomorphism on cohomology. Then the same holds for the morphism $\phi \otimes A$ for any object $A \in \langle \, \{ S^m U^\dual \}_{m \in \Z} \, \rangle$. By Lemma \ref{tautological subcategories} there is a different set of generators for this subcategory, the sequence of exterior powers $\Lambda^m Q^\dual$. For the zeroth exterior power $\Lambda^0 Q^\dual \caniso \O_{\Gr(k, V)}$ the map on cohomology induced by $\phi$ is an isomorphism since (\ref{decomposition for grkn}) is an exceptional collection. For positive exterior powers the cohomology of the source object  $\Lambda^m Q^\dual$ vanishes, and the cohomology of the target is
\[
\RGamma(\Gr(k, V), L_{\mu^T}(Q)^\dual \otimes L_{\mu^T}(Q) \otimes \Lambda^{m} Q^\dual) \caniso \Hom_{\Gr(k, V)}(L_{\mu^T}(Q) \otimes \Lambda^{m} Q, L_{\mu^T}(Q))
\]
which also is zero, as explained in the proof of semiorthogonality. Thus
\[
\begin{aligned}
\MoveEqLeft[6]
\Hom_{\Tot(U)}(\pi^* L_{\mu^T}(Q), \pi^* L_{\mu^T}(Q)) \caniso \RGamma(\Gr(k, V), \pi_*\pi^* \O_{\Gr(k, V)}) \caniso \\
& \caniso \mathrm{R}^0\Gamma(\Gr(k, V), \oplus_{m=0}^{\infty}\, S^m U^\dual)\caniso \oplus_{m=0}^\infty\, S^m V^\dual \caniso \k[\A(V)].
\end{aligned}
\]

In other words, this computation proves that for any $\mu \in B_{k, n-k} \setminus B_{k, n-k}'$ the functor
\[
A \in \Dbcoh(\A(V)) \,\, \mapsto \,\, \pi^* L_{\mu^T}(Q) \otimes p^*(A) \,\, \in \Dbcoh(\Tot(U))
\]
is fully faithful, and its image is equal to $\langle \, \pi^*L_{\mu^T}(Q) \, \rangle$ \cite[(4.3)]{keller94}. Note that $p$ is proper since it factors as a composition $\Tot(U) \hookrightarrow \Gr(k, V) \times \A(V) \to \A(V)$, so by
Lemma \ref{adjoints for proper maps of smooth varieties} 
the functor $p^*$ has both left and right adjoints.  The same is true for tensor multiplication by  the bundle $L_{\mu^T}(Q)$, hence the inclusion functor $\langle \, \pi^*L_{\mu^T}(Q) \, \rangle \hookrightarrow \Dbcoh(\Tot(U))$ is an embedding of an admissible subcategory.

\end{proofstep}

\begin{proofstep}[completeness]
It remains to show that the collection is full. All subcategories are admissible, so we consider the common orthogonal. Let $T \in \Dbcoh(\Tot(U))$ be an object which is left-orthogonal to the first row of (\ref{decomposition for ukn}) and right-orthogonal to the second row of (\ref{decomposition for ukn}), i.e.,
\[
\begin{array}{ll}
\forall \lambda \in B_{k, n-k}' \,\, & \Hom_{\Tot(U)}(T, \,\, j_* L_\lambda(U)) = 0. \\ 
\forall \mu \in B_{k, n-k} \setminus B_{k, n-k}' \,\, & \Hom_{\Tot(U)}(\pi^*L_{\mu^T}(Q), \, T) = 0.
\end{array}
\]

We plan to show that $T$ is zero by considering the related objects $\pi_* T$ and $j^* T$ on the Grassmannian.
The adjunction immediately implies that $\pi_* T$ is right-orthogonal
to each bundle from the $\pi^*$-part. By the completeness of the collection (\ref{decomposition for grkn}) the object $\pi_* T$ lies in the subcategory of $\Dqcoh(\Gr(k, V))$ generated by the $j_*$-part. Similarly, $j^* T$ lies in the subcategory generated by the $\pi^*$-part. 

There is a natural map $\pi_* T \to j^* T$, arising from the canonical morphism $\varphi_T\colon T \to j_* j^*T$ by an application of $\pi_*$. We describe its cone next. The map $\varphi_T$ can also be obtained by multiplying the natural morphism $\O_{\Tot(U)} \to j_* \O_{\Gr(k, V)}$ by $T$. So the cone of $\varphi_T$ is $T \otimes I[1]$, where $I$ is the ideal sheaf of the zero section in $\Tot(U)$. The pushforward of $\varphi_T$ along $\pi$ leads to a triangle  in the category $\Dqcoh(\Gr(k, V))$:
\[
\pi_*(T \otimes I) \to \pi_* T \to j^* T \to \pi_*(T \otimes I[1]).
\]

Recall that there exists a Koszul resolution
\[
j_* \O_{\Gr(k, V)} \caniso [ \,\, \pi^* \Lambda^k U^\dual \to \ldots \to \pi^*\Lambda^2 U^\dual \to \pi^* U^\dual \to \O_{\Tot(U)} \,\, ].
\]
The map $\O_{\Tot(U)} \to j_* \O_{\Gr(k, V)}$ corresponds to the inclusion of the last term. Thus the ideal sheaf $I$ lies in the subcategory generated by $\pi^*\Lambda^{> 0} U^\dual$. Hence the object $C := \pi_*(T \otimes I)$ lies in the subcategory generated by $\pi_* T \otimes \Lambda^{> 0} U^\dual$. As we show below, the existence of this triangle and the semiorthogonality properties of the $j_*$-part imply that $\pi_* T = 0$.

%
%
%

Assume that $\pi_* T \neq 0$. Since $\pi_* T$ lies in the subcategory generated by the $j_*$-part, it cannot be left-orthogonal to it. So let $F$ be a vector bundle from the $j_*$-part which admits a nonzero map $f\colon \pi_* T \to F[i]$ for some shift $i$. We can choose $F$ to be the first vector bundle in the sequence (\ref{decomposition for grkn}) with this property, so $\pi_*T$ lies in the subcategory generated by $F$ and further vector bundles from the $j_*$-part. In part (a) of semiorthogonality we showed that there are no nonzero maps $F \otimes \Lambda^{> 0} U^\dual \to F$ of any degree, and also no maps $F' \otimes \Lambda^{\bullet} U^\dual \to F$ for all further components $F'$ of $\pi_* T$. Consequently, there are no nonzero maps $\pi_*T \otimes \Lambda^{>0}U^\dual \to F$.

So in the diagram below
\[
\begin{tikzcd}
C \arrow[r] \arrow[rd, "", swap] & \pi_* T \arrow[r] \arrow[d, "f"] & j^* T \arrow[r] \arrow[ld, dashed] & C[1] \\
& F[i]
\end{tikzcd}
\]
the solid diagonal arrow is zero and the dashed arrow exists. The object $j^* T$ lies in the subcategory generated by the $\pi^*$-part, $F$ is in the $j_*$-part, so by the semiorthogonality the dashed arrow, and hence the map $f$ itself, must be zero. This contradicts the choice of $F$, so the pushforward $\pi_* T$ is zero. As the morphism $\pi$ is affine, this implies $T = 0$. The common orthogonal to the collection (\ref{decomposition for ukn}) is zero, hence it is complete.\qedhere
\end{proofstep}
\end{proof}

\section{Global situation}
\label{section on relative decomposition for grs}

The situation from Theorem \ref{theorem on decomposition for ukn} may be globalized. Let $X$ be a Cohen--Macaulay variety of dimension $N$, $E$ a vector bundle on $X$ of rank $n$. Let $s \in \Gamma(X, E)$ be a regular section, $Z \subset X$ its zero locus. One way to describe the blow-up of $X$ in $Z$ is as a subvariety of $\P(E)$. Namely, it is the locus of lines in the fibers of $E$ which contain the value of $s$. Over the nonvanishing locus of $s$, the line is uniquely determined, but when $s = 0$, the entire projectivization of the fiber is included. This construction makes sense not only for $\P(E) \caniso \Gr_X(1, E)$, but for all relative Grassmannians. A formal definition may be given as follows. 

\begin{definition}
\label{definition of grs}
Let $q\colon \Gr_X(k, E) \to X$ be the relative Grassmannian, $U$ and $Q$ the relative tautological subbundle and quotient bundle. Let $\tilde{s} \in \Gamma(\Gr_X(k, E), Q)$ be the section arising from $q^*(s) \in \Gamma(\Gr_X(k, E), q^*E)$ and the quotient map $q^*E \to Q$. 
Denote by $\Gr_s(k, E)$ the zero locus of $\tilde{s}$, i.e.,~the locus of $k$-dimensional subspaces which contain the value of $s$.
\end{definition}

By an abuse of notation, the restrictions of tautological bundles on $\Gr_X(k, E)$ to $\Gr_s(k, E)$ are also denoted by $U$ and $Q$. It is easy to see that $\Gr_s(1, E)$ is, indeed, isomorphic to the blow-up of $X$ in $Z$. 

The goal of this section is to show that the category $\Dbcoh(\Gr_s(k, E))$ has a semiorthogonal decomposition into several copies of $\Dbcoh(X)$ and $\Dbcoh(Z)$. The proof of this statement (Theorem \ref{theorem on relative decomposition for grs}) is more or less a formal consequence of Theorem \ref{theorem on decomposition for ukn} and the theory of relative semiorthogonal decompositions, set up in subsection \ref{linear subcategories}. First, let us discuss the relation between Definition \ref{definition of grs} and Theorem \ref{theorem on decomposition for ukn}.

\begin{numberedexample}
\label{tot is a special case of grs}
Let $V$ be a vector space of dimension $N$, $X$ the affine space $\A(V)$, $E$ the trivial bundle $\O_X \otimes V$, and let $s$ be the tautological section. Then $\Gr_s(k, E)$ is a subvariety of the product $\A(V) \times \Gr(k, V)$, consisting of pairs $(v \in V, V_k \subset V)$ such that $v \in V_k$. The projection $\pi\colon \Gr_s(k, \O \otimes V) \to \Gr(k, V)$ identifies $\Gr_s(k, E)$ with the total space of the tautological bundle $\Tot(U_{\Gr(k, V)})$.
\end{numberedexample}

Thus the variety considered in Theorem \ref{theorem on decomposition for ukn} is a special case of Definition \ref{definition of grs}. As will be shown in the proof of Theorem \ref{theorem on relative decomposition for grs}, this special case is a universal example, i.e., any $\Gr_s(k, E)$ flat-locally on $X$ is isomorphic to $\Tot(U) \to \A^n$. Since the derived category of the universal example has already been studied in Theorem \ref{theorem on decomposition for ukn}, it only remains to show that the machinery of relative semiorthogonal decompositions applies to obtain the general result.

The fibers of $\Gr_s(k, E)$ over the points of $Z$ are just $\Gr_Z(k, E|_Z)$. Let $j$, $i$, $p$, $p_Z$ be maps described by the following diagram.

\begin{center}
\begin{tikzcd}
\Gr(k, E|_Z) \arrow[r, "j", hookrightarrow] \arrow[d, "p_Z"] & \Gr_s(k, E) \arrow[d, "p"] \arrow[r, hookrightarrow] & \Gr_X(k, E) \\
Z \arrow[r, "i", hookrightarrow] & X
\end{tikzcd}
\end{center}

\begin{lemma}
\label{maps in grs are lci}
The maps $j, i, p, p_Z$ are projective locally complete intersection morphisms.
\end{lemma}
\begin{proof}
The morphism $i$ is a regular closed embedding as by assumption $Z$ is the zero locus of a regular section $s \in \Gamma(X, E)$. The morphism $p_Z$ is the relative Grassmannian over $Z$. The morphism $p$ factors as $\Gr_s(k, E) \hookrightarrow \Gr_X(k, E) \to X$, where the latter map is smooth and projective, while the former is the zero locus of the section $\tilde{s} \in \Gamma(\Gr_X(k, E), Q)$. Let us show that $\tilde{s}$ is a regular section. 

Since $\Gr_X(k, E)$ is Cohen--Macaulay, we only need to count the dimensions. The preimage of the complement $X \setminus Z$ along $p$ has dimension $N + \dim \Gr(k-1, E/[s]) = N + (k-1)(n-k)$. The preimage of $Z$ along $p$ has dimension $(N - n) + k(n-k)$, a strictly smaller number, so $\dim \Gr_s(k, E) = N + (k-1)(n-k)$. Thus $\codim_{\Gr_X(k, E)}(\Gr_s(k, E)) = n-k = \mathrm{rk}(Q)$, i.e., the section $\tilde{s}$ is regular, hence $p$ is a locally complete intersection morphism.

The morphism $j\colon \Gr_Z(k, E|_Z) \to \Gr_s(k, E)$ is a closed embedding. As shown above, the codimension of this embedding is $k$. Consider the section $p^*(s) \in \Gamma(\Gr_s(k, E), p^*E)$. It vanishes after the composition with $p^*E \to Q$, so it 
lifts to a section $\bar{s} \in \Gamma(\Gr_s(k, E), U)$ of the subbundle $U \subset p^* E$ of rank $k$. It cuts out exactly $\Gr_Z(k, E|_Z)$, so $\bar{s}$ is a regular section, hence $j$ is a regular closed embedding.
\end{proof}

\begin{lemma}
\label{decomposition for ukn is linear}
Let $V, U, Q, \pi, j, p$, $B_{k, n-k}$, $B_{k, n-k}^\prime$ be as in Theorem \textup{\ref{theorem on decomposition for ukn}}.
The semiorthogonal decomposition {\rm (\ref{decomposition for ukn})} of the category $\Dbcoh(\Tot(U))$ constructed in Theorem \textup{\ref{theorem on decomposition for ukn}} comes from a relative semiorthogonal decomposition for the variety $\Tot(U)$ over $\A(V)$.
\end{lemma}
\begin{proof}
By Example \ref{tot is a special case of grs} and Lemma \ref{maps in grs are lci} it is enough to show that the components of the decomposition (\ref{decomposition for ukn}) are images of relative Fourier--Mukai functors with perfect kernels.

Some components of (\ref{decomposition for ukn}) are generated by exceptional objects of the form $j_* L_\lambda(U)$ for Young diagrams $\lambda \in B_{k, n-k}^\prime$. Let $Z = \{ 0 \} $ denote the origin point of $\A(V)$. Let $\Phi_{\lambda}$ be the functor $\Dbcoh(Z) \to \Dbcoh(\Tot(U))$ defined by the formula
\[
C \in \Dbcoh(Z) \left(\caniso D^b(\mathrm{Vect})\right) \, \mapsto \, C \otimes j_* L_\lambda(U).
\]
It is easy to see that $\Phi_\lambda$ is a Fourier--Mukai transform with the perfect kernel $L_\lambda(U)$ on the product $Z \times_{\A(V)} \Tot(U) \caniso \Gr(k, V)$. Since $j_* L_\lambda(U)$ is an exceptional object, the functor $\Phi_\lambda$ is fully faithful and its image coincides with $\langle \, j_* L_\lambda(U) \, \rangle$ (see, e.g., \cite[Th.~6.2]{bondal-exceptional}).

The other kind of subcategories used in (\ref{decomposition for ukn}) are those generated by the objects $\pi^* L_{\mu^T}(Q)$ for Young diagrams $\mu\in B_{k, n-k} \setminus B_{k, n-k}^\prime$. Let $\Psi_{\mu^T}\colon \Dbcoh(\A(V)) \to \Dbcoh(\Tot(U))$ be the functor defined by 
\[
A \in \Dbcoh(\A(V)) \, \mapsto \, p^*(A) \otimes \pi^*L_{\mu^T}(Q).
\]
It is easy to see that $\Psi_{\mu^T}$ is a Fourier--Mukai transform with the perfect kernel $\pi^* L_{\mu^T}(Q)$ on the product $\A(V) \times_{\A(V)} \Tot(U) \caniso \Tot(U)$. The fully faithfulness of $\Psi_{\mu^T}$ has been established during the proof of Theorem \ref{theorem on decomposition for ukn} by computing the endomorphism algebra 
of $\pi^* L_{\mu^T}(Q)$
and applying \cite[(4.3)]{keller94}.
\end{proof}

Now we can state and prove the main theorem of this section.

\begin{theorem}
\label{theorem on relative decomposition for grs}
Let $X$ be a Cohen--Macaulay variety of dimension $N$, $E$ a vector bundle on $X$ of rank~$n$. Let $s \in \Gamma(X, E)$ be a regular section, $Z \subset X$ its zero locus. 
Let $B_{k, n-k}^{\prime}$ denote the subset of Young diagrams with exactly $k$ rows, equipped with any order reversing the inclusions. In the notation of Lemma \textup{\ref{maps in grs are lci}}, define two families of functors:
\begin{enumerate}
\item for each $\lambda \in B_{k, n-k}^\prime$ let $\Phi_\lambda\colon \Dbcoh(Z) \to \Dbcoh(\Gr_s(k, E))$ be the functor 
\[
\Phi_\lambda(F) := j_*(p_Z^* F \otimes L_\lambda(U));
\]
\item for each $\mu \in B_{k, n-k} \setminus B_{k, n-k}^\prime$ let $\Psi_{\mu^T}\colon \Dbcoh(X) \to \Dbcoh(\Gr_s(k, E))$ be the functor
\[
\Psi_{\mu^T}(G) := L_{\mu^T}(Q) \otimes p^*G.
\]
\end{enumerate}
Then the functors $\Phi_\lambda$ and $\Psi_{\mu^T}$ are fully faithful embeddings of admissible subcategories, and there is a relative semiorthogonal decomposition
\[
\label{relative decomposition for grs}
\tag{$\Delta$}
\begin{aligned}
\Dbcoh(\Gr_s(k, E)) \caniso \langle \quad
& \{ \,\, \Phi_\lambda\left( \Dbcoh(Z) \right) \,\, \}_{\lambda \in B_{k, n-k}^\prime} \, , \\
& \{ \,\, \Psi_{\mu^T}\left( \Dbcoh(X) \right) \,\, \}_{\mu \in B_{k, n-k} \setminus B_{k, n-k}^\prime } \quad \rangle .
\end{aligned}
\]
\end{theorem}


\begin{proof}
The functors $\Phi_\lambda$ and $\Psi_{\mu^T}$ are by definition relative Fourier--Mukai transforms with perfect kernels $L_\lambda(U)|_{\Gr_Z(k, E|_Z)} \in \Dperf(\Gr_Z(k, E|_Z))$ and $L_{\mu^T}(Q) \in \Dperf(\Gr_s(k, E))$. The projection morphisms mentioned in Definition \ref{definition of relative sod} are $j$, $p$, $p_Z$, and $\mathrm{id}_{\Gr_s(k, E)}$, which are projective locally complete intersections by Lemma \ref{maps in grs are lci}. So by Theorem \ref{relative descent of sod} we can check the existence of (\ref{relative decomposition for grs}) Zariski-locally.

Let $z \in Z$ be a closed point. Choose an affine open neighborhood $W_z \subset X$ of $z$ 
in $X$ 
such that $E|_{W_z}$ is trivial. The section $s|_{W_z} \in \Gamma(W_z, \O_{W_z}^{\oplus n})$ is locally given by $n$ functions, $f_1, \ldots, f_n$, which define a map $W_z \to \A^n$\!. The section $s$ is regular and vanishes at $z$, so in the local ring of $z$ the stalks of functions $f_1, \ldots, f_n$ form a regular sequence. By \cite[Tag~07DY]{stacks-project} and openness of flat locus, after possibly shrinking $W_z$, we may assume that the map $W_z \to \A^n$ is flat.


By construction $E|_{W_z}$ and $s|_{W_z}$ are pullbacks along $W_z \to \A^n$ of the trivial rank $n$ vector bundle and its tautological section. Recall that the $\Gr_s(k, E)$-construction applied to the tautological section of $\O_{\A^n}^{\oplus n}$ is exactly the total space of the tautological bundle on $\Gr(k, n)$. Thus we obtain a diagram of Cartesian squares:

\begin{center}
\begin{tikzpicture}[node distance=1cm, auto]
  \node (tot) {$\Tot(U_{k, n})$};
  \node (grsuz) [right =of tot] {$\Gr_{s|_{W_z}}(k, E|_{W_z})$};
  \node (grs) [right =of grsuz] {$\Gr_s(k, E)$};
  \node (an) [below =of tot] {$\A^n$};
  \node (uz) [below =of grsuz] {$W_z$};
  \node (x) [below =of grs] {$X$};
  \draw[->] (grsuz) to (tot);
  \draw[->] (grsuz) to (grs);
  \draw[->] (tot) to (an);
  \draw[->] (grsuz) to (uz);
  \draw[->] (grs) to (x);
  \draw[->] (uz) to node [swap] {\scriptsize flat} (an);
  \draw[right hook->] (uz) to (x);
  \node (pullback2) [left =of grsuz, xshift=2.2cm, yshift=-0.5cm] {\Large $\urcorner$};
  \node (pullback3) [right =of grsuz, xshift=-2.2cm, yshift=-0.5cm] {\Large $\ulcorner$};
\end{tikzpicture}
\end{center}

Note that by construction the pullbacks of the tautological bundles from both $\Tot(U_{k, n})$ and~$\Gr_s(k, E)$ to $\Gr_{s|_{W_z}}(k, E|_{W_z})$ coincide. The Fourier--Mukai kernels appearing in Lemma~\ref{decomposition for ukn is linear} and in (\ref{relative decomposition for grs}) are given in terms of the Schur functors of the tautological bundles. Therefore the base change of (\ref{relative decomposition for grs}) to $\Gr_{s|_{W_z}}(k, E|_{W_z})$ coincides with the base change of (\ref{decomposition for ukn}). Since the morphism $W_z \to \A^n$ is flat, the base change of (\ref{decomposition for ukn}) to $\Gr_{s|_{W_z}}(k, E|_{W_z})$ is, by Theorem \ref{relative base change of sod} and Lemma \ref{decomposition for ukn is linear}, a relative semiorthogonal decomposition. Hence the base change of (\ref{relative decomposition for grs}) also is one.

Now let $W = X \setminus Z$ be the nonvanishing locus of $s$. The restriction of $\Gr_s(k, E)$ to $W$ is a relative Grassmannian $\Gr(k-1, E|_W / \langle s \rangle)$ over $W$. The base changes of $\Phi_{\lambda}$ are inclusions of the zero subcategory, while the base changes of $\Psi_{\mu^T}$ form a relative version of the Kapranov's collection on $\Gr(k-1, n-1)$.
This is a full exceptional collection on the relative Grassmannian by \cite[(3.7)]{kapranov88}.

Thus for the open cover consisting of $W$ and the collection of $W_z$ for each point $z\in Z$ all conditions of Theorem \ref{relative descent of sod} are satisfied. Hence the semiorthogonal decomposition descends from the cover to $\Gr_s(k, E)$ itself.
\end{proof}

\begin{corollary}
\begin{enumerate}
\item The functors $\Phi_\lambda$ and $\Psi_{\mu^T}$ restrict to functors $\Phi_{\lambda, \, \mathrm{perf}}$ and $\Psi_{\mu^T, \, \mathrm{perf}}$ between the categories of perfect complexes in a way that
\[
\begin{aligned}
\Dperf(\Gr_s(k, E)) \caniso \langle \quad
& \{ \,\, \Phi_{\lambda, \, \mathrm{perf}}\left( \Dperf(Z) \right) \,\, \}_{\lambda \in B_{k, n-k}^\prime} \, , \\
& \{ \,\, \Psi_{\mu^T, \, \mathrm{perf}}\left( \Dperf(X) \right) \,\, \}_{\mu \in B_{k, n-k} \setminus B_{k, n-k}^\prime } \quad \rangle  .
\end{aligned}
\]
\item The functors  $\Phi_\lambda$ and $\Psi_{\mu^T}$ extend to functors $\Phi_{\lambda, \, \mathrm{qc}}$ and $\Psi_{\mu^T, \, \mathrm{qc}}$ between the unbounded categories of quasi-coherent sheaves in a way that
\[
\begin{aligned}
\Dqcoh(\Gr_s(k, E)) \caniso \langle \quad
& \{ \,\, \Phi_{\lambda, \, \mathrm{qc}}\left( \Dqcoh(Z) \right) \,\, \}_{\lambda \in B_{k, n-k}^\prime} \, , \\
& \{ \,\, \Psi_{\mu^T, \, \mathrm{qc}}\left( \Dqcoh(X) \right) \,\, \}_{\mu \in B_{k, n-k} \setminus B_{k, n-k}^\prime } \quad \rangle  .
\end{aligned}
\]
\end{enumerate}
\end{corollary}
\begin{proof}
Follows from Theorem \ref{theorem on relative decomposition for grs} and Lemma \ref{relative decompositions are for all categories}.
\end{proof}

\printbibliography

\end{document}